\documentclass{article}
\usepackage[utf8]{inputenc}
\usepackage{amsmath}
\usepackage{amssymb}
\usepackage{amsfonts}
\usepackage{amsthm}
\usepackage{graphicx}
\usepackage{appendix}
\usepackage[capitalise,noabbrev,nameinlink ]{cleveref}
\usepackage[colorinlistoftodos]{todonotes}
\usepackage[ruled, linesnumbered]{algorithm2e}

\newtheorem{theorem}{Theorem}

\newtheorem{lemma}{Lemma}

\newtheorem{proposition}{Proposition}

\usepackage{mathtools}
\DeclarePairedDelimiter{\ceil}{\lceil}{\rceil}

\DeclareMathOperator{\poly}{poly}

\newcommand{\Nbh}{\mathcal{N}}
\newcommand{\NbhP}{\Nbh^+}
\newcommand{\NbhM}{\Nbh^-}

\title{The Burning Number of Directed Graphs: Bounds and Computational Complexity}
\author{Remie Janssen\footnote{Delft Institute of Applied Mathematics, Delft University of Technology, Van Mourik Broekmanweg 6,
2628 XE Delft,
The Netherlands. R.Janssen-2@tudelft.nl. Research funded by the Netherlands Organization for Scientific Research (NWO), Vidi grant 639.072.602 of dr. Leo van Iersel.}}
\date{\today}

\begin{document}

\maketitle
\begin{abstract}
The burning number of a graph was recently introduced by Bonato et al. 
Although they mention that the burning number generalises naturally to directed graphs, no further research on this has been done. 
Here, we introduce graph burning for directed graphs, and we study bounds for the corresponding burning number and the hardness of finding this number. 
We derive sharp bounds from simple algorithms and examples. 
The hardness question yields more surprising results: finding the burning number of a directed tree is NP-hard, but FPT; however, it is W[2]-complete for DAGs. 
Finally, we give a fixed-parameter algorithm to find the burning number of a digraph, with a parameter inspired by research in phylogenetic networks.  
\end{abstract}

\section{Introduction}
The burning number of a graph was recently introduced by Bonato et al. \cite{bonato2014burning} as a model of social contagion. 
Social contagion is the spread of rumours, behaviour, emotions, or other social information through social networks (e.g., \cite{bond201261,kramer2014experimental}). 
This information transfer can be active, but may also happen passively, for example by reading posts on social media.

In the discussion of \cite{bonato2016burn}, Bonato et al. remark that the burning number generalises naturally to directed graphs. 
However, no further research on burning directed graphs has been done. 
Nevertheless, the topic seems relevant, for example in the same setting as undirected graphs, social contagion: 
some communication occurs in mostly one direction, such as when people follow other people on social media. 
Hence, in this paper, we introduce the directed version of graph burning and we study some basic problems related to graph burning.

The central concept in graph burning is the burning number. 
The burning number of a graph is the number of steps it takes to ``burn'' this graph. 
Here, burning is a step-wise process roughly defined as follows. 
In each step, all neighbours of all burned nodes are burned, and one extra node is chosen that is burned as well. 
This models social contagion in the sense that neighbours start to burn correspond to information being spread through the network. 
The basic questions we study relate to bounds on the burning number for different classes of graphs, and to the computational complexity of determining the burning number of a graph. 
Both of these problems have been studied for undirected graphs \cite{land2016upper, bessy2018bounds, bonato2019bounds, bessy2017burning}.

\paragraph{Structure of the paper.} 
We start by defining our notation, and basic concepts in graph burning in \cref{sec:Preliminaries}. 
Then, in \cref{sec:Bounds}, we give bounds on the burning number for weakly connected digraphs, single sourced DAGs, and strongly connected graphs. 
Following this, in \cref{sec:Complexity}, we turn to the complexity of the graph burning problem in directed graphs. 
We show that the problem is NP-hard for trees, but solvable in FPT time; and for DAGs, we show that the problem is W[2]-hard. 
We finish this section with an algorithm for digraph burning. 
Finally, in \cref{sec:Discussion}, we discuss the relation between directed graph burning, undirected graph burning and some other problems.

\section{Preliminaries}\label{sec:Preliminaries}
The burning number of directed graphs is defined analogously to that of undirected graphs, with one adaptation: burning only occurs in the direction of arcs. 
This leads to the following definition of digraph burning. 

A sequence $(x_1,\ldots,x_b)$ of nodes is a \emph{burning sequence} for a digraph $D$ if after $b$ steps of the following process, every node of $D$ is \emph{burned}. 
The $i$-th step consists of burning all out-neighbours of all currently burning nodes, and then burning the node $x_i$. 
Note that burning of neighbours occurs before burning the chosen node, so after the first step, only the chosen node is burned. 
The \emph{burning number} of a digraph is the length of the shortest burning sequence for that graph. This leads to the following computational problem.

\medskip
\fbox{
\parbox{0.8\linewidth}{
{\sc Digraph Burning}\\
{\bf Input:} A digraph $D$.\\
{\bf Parameter:} A natural number $b$.\\
{\bf Output:} A burning sequence of $D$ of length at most $b$ if it exists; NO otherwise.}
}
\medskip

When we restrict the input to smaller classes of graphs, like trees or DAGs, we change the name of the problem accordingly, so it becomes {\sc Tree Burning} resp. {\sc DAG Burning}.

Like for undirected graphs, the directed burning number can alternatively be defined using covers \cite{bonato2016burn} of neighbourhoods.
The \emph{$k$-out-neighbourhood} $\NbhP_k(v)$ of a node $v$ consists of all nodes that can be reached from $v$ using at most $k$ arcs. 
Similarly, the \emph{$k$-in-neighbourhood} $\NbhM_k(v)$ of $v$ consists of all nodes from which $v$ can be reached using at most $k$ arcs. Note that $\NbhM_0(v)=\NbhP_0(v)=\{v\}$.
A digraph has burning number $b$ if and only if there is a sequence of $b$ nodes $(v_1,\ldots,v_b)$ such that 
\[D=\bigcup_{i=1}^b\NbhP_{i-1}(v_i).\]
The corresponding burning sequence is $(v_b,\ldots,v_1)$, with the minor caveat that a burning node cannot be chosen in any step, but if the cover does this, an arbitrary non-burning node can be chosen instead.\\
\\
For algorithmic purposes, we also consider variants of the directed burning problem. 
For recursion, we need to consider instances where parts of the graph are already burned, and, hence, do not need to be burned anymore; 
and we need to consider instances where we are no longer allowed to choose a node to burn in each step.

These variations naturally lead to the concepts of partial burning---first defined for undirected graphs in \cite{bonato2019improved}---and burning ranges. 
For the first, we no longer ask for the whole graph to be burned at the end of the process, but only a selected subset of nodes. 
In other words, the graph only needs to be partially burned. 
For the second, it is convenient to consider the covering definition of the burning number where the sequence is replaced by a map. That is, the burning number is at most $b$ if there is a \emph{burning assignment} $\phi:[b]\to V(D)$ such that
\[D=\bigcup_{i=1}^b\NbhP_{i-1}(\phi(i)).\]
In this formulation, we call the set $[b]=\{1,\ldots,b\}$ the set of \emph{burning ranges}. 
Of course, the set of burning ranges can be any set of natural numbers. 
This leads to the following generalization of the burning problem.

\medskip
\fbox{
\parbox{0.8\linewidth}{
{\sc Partial Digraph Burning With Ranges}\\
{\bf Input:} A Digraph $D=(V,A)$, a set $X\subseteq V$, and a set of burning ranges $R$.\\
{\bf Output:} A burning assignment $\phi:R\to V$ such that $X\subseteq\bigcup_{r\in R}\NbhP_{r-1}(\phi(r))$ if it exists; NO otherwise.}
}
\medskip


\section{Bounds for the directed burning number}\label{sec:Bounds}
For undirected graphs, people have attempted to prove the conjecture that any connected graph with $n$ nodes has burning number at most $\ceil{\sqrt n}$. Therefore, the worst case for undirected graphs seems to be the path, on which this bound is attained. Here, we investigate similar bounds for directed graphs with several connectivity assumptions. We start by showing that not assuming any connectivity makes the problem quite uninteresting. Later, we show that for directed graphs, the directed path is the worst case among al sufficiently connected graphs.

\begin{lemma}\label{clm:StupidBoundDisconnected}
Let $D$ be a digraph with $n$ nodes. If $D$ contains no arcs, or $n=1,2$, the burning number is equal to $n$.
If $n>2$ nodes and $D$ has at least one arc, the burning number is at most $n-1$, and this bound is sharp. 
\end{lemma}
\begin{proof}
The first part of the lemma is obvious, so we assume $D$ has at least one arc, and more than $2$ nodes. We first burn the tail end of an arc. The head of this arc will burn in the next time step. In the subsequent time steps, we choose to burn all other nodes. This takes $n-1$ steps in total. The bound is sharp, as we can take the graph with $n-2$ isolated vertices, and the remaining $2$ nodes connected by an arc.
\end{proof}

This proof relies on graphs being poorly connected. As this makes no sense in the setting of social contagion, we now restrict to weakly connected graphs.

\begin{lemma}
For weakly connected digraphs with $n>2$ nodes, the burning number is at most $n-1$, and this bound is sharp.
\end{lemma}
\begin{proof}
As weakly connected graphs with $n>2$ nodes have at least one arc, \cref{clm:StupidBoundDisconnected} applies and the bound of $n-1$ holds. The bound is still sharp: take the digraph consisting of $n-1$ sources, all with an arc pointing to the remaining node. As each source has to be part of the burning sequence, this graph has burning number at least $n-1$.
\end{proof}

In this case, the number of sources makes it easy to force a high burning number. Therefore, we now limit the number of sources as well, before turning to strongly connected digraphs.

\subsection{Single source DAGs}
We start by computing the burning number of the simplest single source DAG, the directed path. Then, we show that paths have the largest burning number, in the sense that among all single source DAGs with $n$ nodes, the burning number of the path is the largest. 

\begin{lemma}
The burning number of the directed path on $n$ nodes, $P_n$, is \[\ceil[\Big]{\sqrt{2n+\frac{1}{4}}-\frac{1}{2}}.\]
\end{lemma}
\begin{proof}
In $t$ time steps, we can burn $t(t+1)/2$ nodes of the path. Hence, to burn a path of length $n$, we need $t_n$ steps such that $n\leq t_n(t_n+1)/2$, or $t_n\geq \sqrt{2n+\frac{1}{4}}-\frac{1}{2}$. As the burning number is integral, we take the ceiling to get the burning number.
\end{proof}

To prove that all single source DAGs with $n$ nodes have burning number at most $\ceil[\Big]{\sqrt{2n+1/4}-1/2}$, we use the concept of eccentricity.
The \emph{eccentricity} of a node $v$ is the maximal distance to any node that can be reached from $v$. Note that we only look at the distances to nodes that can be reached. Hence, the eccentricity of a node is always a finite number.

\begin{lemma}\label{clm:eccentricity}
Let $D$ be a single source DAG, and let $\rho$ be its source with eccentricity $e$. For all $i\in[e]$, there is a node in $D$ with eccentricity $i$.
\end{lemma}
\begin{proof}
We prove the claim using induction on the eccentricity from high to low. For the basis, note that the source has eccentricity $e$, so there exists a lowest node of eccentricity $e$. 

Now assume there is a node $x_i$ with eccentricity $i\in[e]$ with no node of eccentricity at least $i$ below it. We prove there exists a node with eccentricity $i-1$ with no node of eccentricity at least $i-1$ below it. Note that at least one of the children $c$ of $x_i$ has eccentricity at least $i-1$, otherwise the shortest distance of $i$ from $x_i$ to some other node cannot be achieved. As $x_i$ there is no node below $x_i$ with eccentricity $i$ or greater, $c$ has to have eccentricity $i-1$. Now look for the lowest node with eccentricity $i-1$ below $c$ (including $c$). This node has no nodes of eccentricity $i-1$ or greater below it. 

Therefore, for each $i\in[e]$, there is a node with eccentricity $i$ and no node of eccentricity at least $i$ below it. In particular, for each $i\in[e]$, there is a node with eccentricity $i$, which is what we needed to prove.
\end{proof}

\begin{proposition}\label{clm:BurningSingleSourceDAG}
The directed burning number of a single source DAG is at most $\ceil[\Big]{\sqrt{2n+\frac{1}{4}}-\frac{1}{2}}$.
\end{proposition}
\begin{proof}
For convenience in notation, we define $b_n$ as the smallest solution for $n\leq \sum_{i=1}^b i$, i.e., $b_n:=\ceil[\Big]{\sqrt{2n+\frac{1}{4}}-\frac{1}{2}}$. We use induction on the number of nodes in the graph to prove the claim. For the induction basis, we say the empty graph has burning number $0$. 

Now assume each single source DAG with $n<N$ nodes has burning number at most $b_n$. Consider a graph $D$ with $N$ nodes. First suppose the root source $s$ has eccentricity at most $b_n-1$. In that case, $D=\NbhP_{b_n-1}(s)$, and the burning number is at most $b_n$. Now suppose the source has eccentricity greater than $b_n-1$, then there exists a node $v$ of eccentricity $b_n-1$ (\cref{clm:eccentricity}). The neighbourhood $\NbhP_{b_n-1}(v)$ contains at least $b_n$ nodes, and $D'=D\setminus \NbhP_{b_n-1}(v)$ is again a single source DAG. As $b_n$ is the smallest solution to $n\leq \sum_{i=1}^b i$, we have $b_{n-b_n}\leq b_n-1$. Hence, by the induction hypothesis, $D'$ has burning number at most $b_n-1$. Taking the corresponding cover of $D'$, and adding $\NbhP_{b_n-1}(v)$, we see that $D$ has burning number at most $b_n$.
\end{proof}

\subsection{Strongly connected graphs}

\begin{lemma}\label{clm:CycleBurning}
The burning number of the directed cycle on $n$ nodes, $C_n$, is $\ceil[\Big]{\sqrt{2n+\frac{1}{4}}-\frac{1}{2}}$.
\end{lemma}
\begin{proof}
The same as the proof for the directed path.
\end{proof}

\begin{theorem}
The burning number of a strongly connected graph is at most $\ceil[\Big]{\sqrt{2n+\frac{1}{4}}-\frac{1}{2}}$, and this bound is sharp.
\end{theorem}
\begin{proof}
Take an arbitrary vertex $v$ of a strongly connected graph $G$ with $n$ nodes, and consider the subgraph with root $v$ and the arcs from all shortest paths from $v$ to all other nodes. This is a single source DAG with $n$ nodes, so it has burning number at most $\ceil[\Big]{\sqrt{2n+\frac{1}{4}}-\frac{1}{2}}$ by \cref{clm:BurningSingleSourceDAG}. This implies that the burning number of $G$ is also at most $\ceil[\Big]{\sqrt{2n+\frac{1}{4}}-\frac{1}{2}}$. By \cref{clm:CycleBurning}, this bound is sharp.
\end{proof}

\section{Complexity of computing the directed burning number}\label{sec:Complexity}
In this section, we consider the complexity of the directed graph burning problem. It is easy to see that the problem is NP-hard in general. Indeed, a reduction from the undirected version---which is NP-hard as well \cite{bessy2017burning}---can easily be given: take an undirected graph, and replace each edge by two arcs in both directions. Although this shows hardness, it does not highlight the differences between the directed and the undirected versions of the graph burning problem. 

Here, we restrict to graphs that are not strongly connected: trees and DAGs. We show that the problem remains NP-hard, even if we restrict to trees, but that this problem can be solved in FPT time. Then, we turn to DAGs, for which we show the problem is W[2]-complete with the burning number as parameter, which indicates there is probably no fixed parameter algorithm with this parameter. Finally, we give two algorithms that, considering the previous, solve the problem relatively efficiently.

\subsection{Tree Burning}
In this subsection, we show that the directed burning problem is NP-hard, even when considering only trees. Like for undirected graph burning \cite{bessy2017burning}, the proof uses a reduction from {\sc Distinct 3-Partition}, which is a strongly NP-hard problem \cite{hulett2008multigraph}.

\medskip
\fbox{
\parbox{0.8\linewidth}{
{\sc Distinct 3-Partition}\\
{\bf Input:} A finite set $A={a_1, a_2,\ldots, a_{3n}}$ of positive distinct integers, and a positive integer $B$ where $\sum_{i=1}^{3n}a_i=nB$, and $B/4<a_i<B<2$, for $1\leq i \leq 3n$.\\
{\bf Output:} A partition of $A$ in $n$ triples, such that the elements of each triple sum up to $B$; NO otherwise.}
}
\medskip

The reduction in the proof of the following lemma constructs a tree from a sequence of small components. These components are paths $P_n$ on $n$ nodes, and \emph{2-legged spiders} $\Lambda_n$ constructed from two copies of $P_n$ by identifying their source nodes.

\begin{lemma}
{\sc Tree Burning} is NP-complete.
\end{lemma}
\begin{proof}
{\sc Tree Burning} is in NP as a sequence of burning nodes is a certificate that can be checked in polynomial time, so we focus on NP-hardness.

We reduce from {\sc 3-Partition}. Let $A=\{a_i\}_{i=1}^{3n}$ be an instance of {\sc 3-Partition} with $B=\frac{1}{n}\sum a_i$ and $m=\max A$, and let $Z=[m]\setminus A=\{z_i\}_{i=1}^{m-3n}$ be the set of integers up to $m$ not contained in $A$. Now create a directed graph from the following sequence of graphs by attaching the source of each graph to a source of the preceding graph with an arc between these nodes:
\[\Lambda_{z_1},\Lambda_{m+1},\ldots,\Lambda_{z_{|Z|}},\Lambda_{m+|Z|},P^1_{B},\Lambda_{m+|Z|+1},P^2_{B},\ldots,P^{n-1}_{B},\Lambda_{m+|Z|+n-1},P^n_{B}\]
Note that each part corresponding to an element of $Z$ or of the $n$ copies of $P_B$ is followed by a 2-legged spider. The instance of {\sc Tree Burning} consists of this graph, and the integer $m+|Z|+n-1$. Note that the graph has less than $2(1+\cdots+m+|Z|+n-1)<2(1+\cdots+3m)=O(m^2)$ edges. Because {\sc 3-Partition} is strongly NP-hard, we may assume the instance size is polynomial in $B$. Hence, as $m<B$, the reduction can be performed in polynomial time.

We now prove the reduction provides a yes-instance of {\sc Tree Burning} iff the instance of {\sc Distinct 3-Partition} is a yes-instance. First, note that the main path has length exactly $1+\cdots+(m+|Z|+n-1)$. Hence, to burn the graph in exactly $m+|Z|+n-1$ steps, all the burning neighbourhoods must be disjoint, and all chosen burning nodes must lie on this path. Note that the legs of the $\Lambda_k$ that do not correspond to elements of $Z$ are of length $m+1,\ldots,m+|Z|+n-1$: all longest ranges. Hence, they must be burned at a time step such that this leg burns exactly to the end. This implies that each top node of a $\Lambda_k$ must be burned in step $m+|Z|+n-k$. The remaining lengths are exactly the $a_i$, and they can cover the $P_B$s exactly if and only if the {\sc 3-Partition} instance is solvable (a partition cannot contain parts of size other than three, because $B/4<a_i<B<2$, for $1\leq i \leq 3n$).
\end{proof}

\begin{algorithm}[H]
 \KwData{A Tree $T$ and a set of allowed burning ranges $R$.}
 \KwResult{A map $\phi:R\to V(T)$ if $T$ can be burned by using $R$, NO otherwise.}
Let $s$ be a sink of $T$ furthest from the source $\rho$\;
\If{$d(\rho,s)\leq \max(R)-1$}{
 \Return $\{r\mapsto\rho: r\in R\}$ \;
}
\For{$r$ in $R$}
{
     Let $v$ be the node with $d(v,s)=r-1$\;
     Let $T'$ be $T$ with all nodes of distance $r-1$ away from $v$ removed\;
     Calculate {\sc TreeBurning}$(T',R\setminus\{r\})$\;
     \If{the result is a map $\phi'$}{
       let $\phi$ be the map obtained from $\phi'$ by adding $r\mapsto v$\;
       \Return $\phi$\;
     }
}
\Return NO\;
\caption{{\sc TreeBurning}$(T,R)$}\label{alg:TreeBurning}
\end{algorithm}

\begin{lemma}\label{clm:GoodAssignmentTree}
Let $T$ be a tree, $R$ a set of burning ranges which can burn $T$, and $s$ a sink furthest from the source $\rho$. Then, either $d(\rho,s)\leq\max(R)-1$, or there exists a burning assignment such that, for some $r\in R$, the node $v$ with $d(v,s)=r-1$ is burned with range $r$.
\end{lemma}
\begin{proof}
First note that if $d(\rho,s)\leq\max(R)-1$, then assigning $\max(R)$ to $\rho$ and assigning all other ranges arbitrarily gives a burning assignment that burns $T$. So, for the rest of the proof, we assume $\max(R)-1<d(\rho,s)$.

Let $\phi:R\to V(T)$ be a burning assignment, and let $v$ be a node with $\phi(r)=v$ and $s\in\NbhP_{r-1}(v)$. If $d(v,s)=r-1$ or $v=\rho$, then the statement of the lemma holds. Hence, suppose that $d(v,s)<r-1$. As $d(v,s)<r-1\leq\max(R)-1<d(\rho,s)$, there is a node $u$ on the path from $\rho$ to $v$ with $d(u,s)=r-1$. Because the distance from $\rho$ to $s$ is maximal among all sinks, $\NbhP_{r-1}(x)=\NbhP_{d(v,s)}(x)$ for all $x$ above $s$ with $d(x,s)\leq r-1$, so in particular, the same holds for $u$. As $\NbhP_{r-1}(v)\subseteq\NbhP_{r-1}(u)$, reassigning $r$ to $u$ gives a new burning assignment of $R$ that burns $T$ for which the condition of the lemma holds.
\end{proof}

\begin{theorem}
Algorithm~\ref{alg:TreeBurning} solves {\sc Tree Burning} in $O(b! \poly(n))$ time.
\end{theorem}
\begin{proof}
We claim the algorithm can decide whether the burning number of a tree $T$ is at most $b$ by calling {\sc TreeBurning}$(T,[b])$. 

By \cref{clm:GoodAssignmentTree}, we have the following: if a tree can be burned with the given ranges $R$, then it can either be burned with one range at the root, or there is a solution that does not overshoot the furthest sink. The algorithm recursively searches for solutions of that type, and will therefore find a solution if it exists. 

For the running time, note that each time the algorithm branches, the set of ranges reduces by one. This means there are at most $b!$ recursive calls of the algorithm in total. Each call takes polynomial time, so the running time is $O(b! \poly(n))$.
\end{proof}

\subsection{DAG Burning}
In the previous section, we saw that {\sc Tree Burning} is NP-hard, but solvable in FPT time. One could hope that an fixed-parameter algorithm also exists for {\sc DAG Burning} with the burning number as the parameter. However, as we will show in this section, this is quite unlikely, as {\sc DAG Burning} is W[2]-complete. This implies that the increased hardness is a consequence of nodes with higher in-degree. Hence, we give an fixed-parameter algorithm for {\sc DAG Burning} with a parameter combining the burning number and the `extra indegree'.

\medskip
\fbox{
\parbox{0.8\linewidth}{
{\sc Set Cover}\\
{\bf Input:} A universe $\mathcal{U}$ and a set $\mathcal{S}$ of subsets $S\subseteq \mathcal{U}$ with $\cup_{S\in\mathcal{S}}S=\mathcal{U}$.\\
{\bf Parameter:} A natural number $k$, the number of subsets in the cover.\\
{\bf Output:} A cover $C\subset\mathcal{S}$ of $\mathcal{U}$ with $|C|\leq k$ if it exists; NO otherwise.}
}
\medskip

To prove W[2]-hardness, we give a parameterized reduction from {\sc Set Cover} to {\sc DAG Burning}, and to prove the problem is in W[2], we give a parameterized reduction in the other direction. As {\sc Set Cover} is W[2]-complete (\cite{cygan2015parameterized} Theorem~13.21), this suffices to show that {\sc DAG Burning} is W[2]-complete as well.

\begin{lemma}
The {\sc DAG Burning} problem is W[2]-hard with the burning number as parameter.
\end{lemma}
\begin{proof}
We give a parameterized reduction from {\sc Set Cover}. Let $(\mathcal{U},\mathcal{S},k)$ be an instance of {\sc Set Cover}, we construct an instance $(D,b)$ of {\sc DAG Burning} as follows. We set $b=k+2$, and $D$ to the DAG consisting of the nodes
\[\left\{v_A^1,\ldots,v_A^{k+1}\right\}_{A\in X} \cup \left\{v_u\right\}_{u\in\mathcal{U}}\cup\{\rho,v_{(\rho,1)}^{k+2}\},\]
and the edges
\[\left\{(v_S,v_u): u\in S\right\}_{S\in\mathcal{S}}\cup\left\{(v_A^i,v_A^{i+1}): i\in[k]\right\}_{A\in X}\cup\left\{(\rho,v_{A}^1)\right\}_{A\in X}\cup\{(v_{(\rho,1)}^{k+1},v_{(\rho,1)}^{k+2})\},\]
where $X=\{(\rho,j)\}_{j=1}^{b}\cup \mathcal{S}$ (e.g., \cref{fig:DAGBurning_SetCover}). This graph can be constructed in polynomial time, and $b(k)=k+2$ is a computable function. Hence, the only thing left to check is that $(D,b)$ is a yes-instance of {\sc DAG Burning} if and only if $(\mathcal{U},\mathcal{S},k)$ is a yes-instance of {\sc Set Cover}.

First note that, to burn all paths consisting of the $v_{(\rho,j)}^i$ nodes, the source $\rho$ needs to be burned in the first time step. Otherwise, the burning sequence needs to contain a node in each of these paths. Then, the remaining node $v_{(\rho,1)}^{k+2}$ can be burned in the last step by adding it in the burning sequence; doing it earlier does not gain any extra burned nodes, so a burning sequence can be assumed to do this. Lastly, we have the second to the $(k+1)$-th element of the burning sequence. To burn the nodes $\{v_u\}_{u\in\mathcal{U}}$, these remaining elements have to be chosen from the $v_{S_j}^i$; in fact, we may choose to pick them from the $v_{S_j}^{k+1}$, as choosing higher nodes will not result in any more nodes being burned in the end. It is easy to see that a choice of $k$ of these nodes corresponds to a choice of $k$ subsets from $\mathcal{S}$, and that all nodes $v_u$ are burned iff these corresponding subsets form a cover of $\mathcal{U}$. Therefore, $(D,b)$ is a yes-instance of {\sc DAG Burning} if and only if $(\mathcal{U},\mathcal{S},k)$ is a yes-instance of {\sc Set Cover}.
\end{proof}

\begin{figure}
    \centering
    \includegraphics[width=\textwidth]{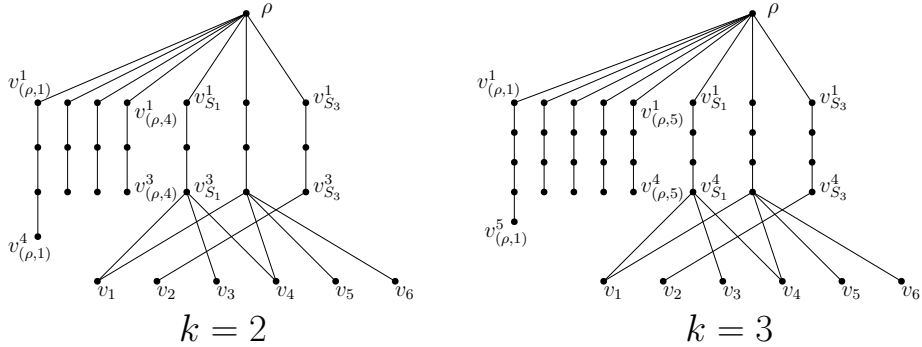}
    \caption{DAG burning W[2]-hardness reduction. The constructed instances of {\sc DAG Burning} for the {\sc Set Cover} instances $([6],\mathcal{S},2)$ and $([6],\mathcal{S},3)$ with $\mathcal{S}=\{\{1,3,4\},\{1,4,5,6\},\{2\}\}$. All arcs are directed down. The first is a no-instance, the second a yes-instance.}
    \label{fig:DAGBurning_SetCover}
\end{figure}

\begin{lemma}
The {\sc DAG Burning} problem is in W[2] with the burning number as parameter.
\end{lemma}
\begin{proof}
We give a parameterized reduction to {\sc Set Cover}. Let $(D,b)$ be an instance of {\sc DAG Burning}, we construct an instance $(\mathcal{U},\mathcal{S},k)$ of {\sc Set Cover} as follows. Let $\mathcal{U}$ be the set consisting of all nodes of $D$ together with a nodes $x_i$ for each $i\in[b]$. The subsets $\mathcal{S}$ are the burning neighbourhoods of all nodes, i.e.,
\[\left\{\NbhP_{i-1}(v)\cup\{x_i\}: v\in V(D), i\in[b] \right\},\]
lastly, we set $k:=b$. As neighbourhoods can be calculated in polynomial time, and there are $nb<n^2$ neighbourhoods to compute, the instance can be produced in polynomial time. Hence, the remaining part is to prove that $(D,b)$ is a yes-instance of {\sc DAG Burning} if and only if $(\mathcal{U},\mathcal{S},k)$ is a yes-instance of {\sc Set Cover}.

To cover the set $V(D)\cup\{x_i\}_{i=1}^b$ with $b$ sets, we must choose exactly one set containing $x_i$ for each $i$.
This corresponds to burning the $i$-neighbourhood for a chosen node for each $i=1\ldots,b$.
That is, if there is a set cover, then we can burn $D$ in $b$ steps.
The other direction is analogous: if we can burn $D$ in $b$ steps, then the corresponding sets will be a set cover.
Therefore, $(D,b)$ is a yes-instance of {\sc DAG Burning} if and only if $(\mathcal{U},\mathcal{S},k)$ is a yes-instance of {\sc Set Cover}, and we have a parameterized reduction from {\sc DAG Burning} to {\sc Set Cover}, which proves that {\sc DAG Burning} is in W[2].
\end{proof}

\begin{theorem}
The {\sc DAG Burning} problem is W[2]-complete with the burning number as parameter.
\end{theorem}

This theorem implies there is probably no fixed-parameter algorithm for {\sc DAG Burning} with the burning number as parameter. However, this does not exclude the possibility that there are fixed-parameter algorithms with a slightly larger parameter, or non fixed-parameter algorithms that have relatively efficient running time when the burning number is small. An algorithm in the last category is Algorithm~\ref{alg:BruteForceBurning}, an adapted version of the guessing algorithm from \cite{kamali2019burning}. In Theorem~2 in that paper, they essentially give an algorithm with running time $O(n^b\poly(n))$, which is polynomial time when the burning number is bounded by a constant. In their case, this happens when the diameter is bounded by a constant (as the burning number is at most the diameter).

\begin{algorithm}[H]
 \KwData{A digraph $D=(V,A)$ and an integer $b$.}
 \KwResult{YES if $b(D)\leq b$, NO otherwise.}
\For{$S\in V^b$}
{
  \If{$S$ is a burning sequence for $D$}
  {
    \Return YES\;
  }
}
\Return NO\;
\caption{{\sc BruteForceBurning}$(D,b)$}\label{alg:BruteForceBurning}
\end{algorithm}

\begin{proposition}
Algorithm~\ref{alg:BruteForceBurning} solves {\sc Digraph Burning} in $O(n^b \poly(n))$ time.
\end{proposition}

For a more sophisticated approach, we try to extend the parameter to include the `extra indegree', which makes {\sc DAG Burning} hard compared to {\sc Tree Burning}. The extra indegree is inspired by the reticulation number of a phylogenetic network (see, for example, \cite{janssen2018exploring}), and is defined as follows. The \emph{reticulation number} $k(D)$ of a digraph $D=(V,A)$ is $\sum_{v\in V}\max(0,\delta^-(v)-1)$. The following algorithm is a fixed-parameter algorithm for {\sc Digraph Burning} with $(k(D)+1)b$ as parameter.

\begin{algorithm}[H]\label{alg:DigraphBurning}
 \KwData{A digraph $D$, a subset $X$ of nodes to be burned, and a set of allowed burning ranges $R$.}
 \KwResult{A map $\phi:R\to V(D)$ if $X$ can be burned by using $R$, NO otherwise.}
Let $x\in X$ be arbitrary\;
\For{$r$ in $R$}
{
  \For{$v\in \NbhM_{r-1}(x)$}{
     Set $X':=X\setminus\NbhP_{r-1}(v)$\;
     Calculate {\sc DigraphBurning}$(D,X',R\setminus\{r\})$\;
     \If{the result is a map $\phi'$}{
       let $\phi$ be the map obtained from $\phi'$ by adding $r\mapsto v$\;
       \Return $\phi$\;
     }
  }
}
\Return NO\;
\caption{{\sc DigraphBurning}$(D,X,R)$}
\end{algorithm}

\begin{lemma}\label{clm:PDBWR}
Algorithm~\ref{alg:DigraphBurning} solves {\sc Partial Digraph Burning With Ranges} in $O(((k(D)+1)\max(R)|R|)^{|R|}\poly(n))$ time.
\end{lemma}
\begin{proof}
We first show the algorithm is correct. Note that each node in $X$ has to burned by some range. Hence, we can pick an arbitrary node $x\in X$, and try each possible assignment of a burning range $r\in R$ to a node $v\in V$ such that $x\in\NbhP_{r-1}(v)$. If there exists a burning assignment, then one of these choices has to be part of such an assignment. For each of these choices, the algorithm checks whether the remaining problem $(D,X\setminus\NbhP_{r-1}(v), R\setminus\{r\})$ has a solution. If it does, then this solution together with $r\mapsto v$ is a solution to the original problem. If no such solution exists for any choice, then the original problem is not solvable, and the algorithm correctly returns NO.

Now we consider the running time. For any fixed range $r$, the size of the neighbourhood $\NbhM_{r-1}(v)$ is at most $(k(D)+1)\max(R)$. For each of the nodes in $\NbhM_{r-1}(v)$, we have to try at most $|R|$ burning ranges. Branching like this, we have to go to depth $|R|$ in the search tree to try all options. Each of the recursive calls takes $\poly(n)$ time for computing the neighbourhoods and the new instances. Hence, the total running time is $O(((k(D)+1)\max(R)|R|)^{|R|}\poly(n))$.
\end{proof}

\begin{theorem}
Algorithm~\ref{alg:DigraphBurning} solves {\sc Digraph Burning} in $O(((k(D)+1)b^2)^b\poly(n))$ time.
\end{theorem}
\begin{proof}
A direct consequence of \cref{clm:PDBWR} when setting $X=V(D)$ and $R=[b]$, in which case $r=|R|=b$.
\end{proof}

Although the running time bound seems quite high, there may be heuristic improvements. For example in Line~1, where the algorithm chooses a node $x\in X$ arbitrarily. Of course, this is not always optimal: one can choose a node $x$ to minimize the size of the neighbourhood $\NbhM_{r-1}(x)$. This makes sure the number of branches is much lower than the theoretical bound of $(k(D)+1)\max(R)|R|$. For example, when $x$ is a source, the number of branches is $|R|$.

\section{Discussion}\label{sec:Discussion}
In this paper, we have studied the directed burning problem. We have given sharp upper bounds on the burning number for digraphs in general ($n-1$), for weakly connected graphs ($n-1$), for single source DAGs ($\ceil{\sqrt{2n+1/4}-1/2}$), and for strongly connected graphs ($\ceil{\sqrt{2n+1/4}-1/2}$). The simple proofs for these bounds contrasts sharply with the undirected case, where the question whether all graphs have burning number $\ceil{\sqrt{n}}$ remains open.

Then, we turned to the computational complexity of computing the burning number. We have shown that computing the burning number of a directed tree is NP-hard. This result may be surprising, as the direction on the trees seems to severely restrict the number of ways one can burn leaves. However, considering that graph burning is hard on undirected trees as well, the result might be less surprising. Another new result for trees is that computing the burning number of a directed tree is FPT with the burning number as parameter. A similar result has not been proven for undirected trees yet, so it would be interesting to see if our results can be leveraged in the undirected version of the problem.

As may be expected, graph burning gets harder when more general graphs are considered: When the burning problem is considered on DAGs with the burning number as parameter, the complexity jumps to W[2]-complete, which we showed with reductions to and from {\sc Set Cover}. This implies there probably is no fixed-parameter algorithm for directed graph burning with the burning number as parameter. Whether there is such an algorithm for undirected graph burning is still an open problem \cite{kare2019parameterized}. Again, it might be interesting to try to use our techniques or results in the undirected case; either by finding parameterized reductions to and from {\sc Set Cover}, or directly to and from {\sc DAG Burning}.

Finally, we have shown that there is an fixed-parameter algorithm for directed graph burning with a parameter consisting of the burning number and the reticulation number, a parameter inspired by phylogenetics research. The algorithm uses recursion, which meant we had to solve the more general problem {\sc Partial Digraph Burning With Ranges}. It is worth mentioning that this problem not only contains {\sc Digraph Burning}, but also a directed version of {\sc Dominating Set}---consider the instance $(D,V(D),\{2,2,\ldots,2,2\})$ with $k$ twos for the {\sc Dominating Set} instance $(D,k)$---and maybe more graph covering problems. 

To really understand the computational complexity of the burning problem, it may be worth studying more restricted problems of {\sc Partial Digraph Burning With Ranges}. For example, what happens if we look for the partial burning number of a DAG where we need to burn only the sinks? Or one could try bounding other graphs parameters, such as the tree-length (or scan-width for DAGs, \cite{berry2020scanning}) used for undirected burning in \cite{kamali2019burning}, and look for approximation algorithms.

\paragraph{Acknowledgements.} The author would like to thank Josse van Dobben de Bruyn for bringing graph burning to his attention, and Mark Jones for lending his expertise in parametrized complexity.

\bibliographystyle{alpha}
\bibliography{Bibliography}

\newcommand{\etalchar}[1]{$^{#1}$}
\begin{thebibliography}{BEKM19}

\bibitem[BBJ{\etalchar{+}}17]{bessy2017burning}
St{\'e}phane Bessy, Anthony Bonato, Jeannette Janssen, Dieter Rautenbach, and
  Elham Roshanbin.
\newblock Burning a graph is hard.
\newblock {\em Discrete Applied Mathematics}, 232:73--87, 2017.

\bibitem[BBJ{\etalchar{+}}18]{bessy2018bounds}
St{\'e}phane Bessy, Anthony Bonato, Jeannette Janssen, Dieter Rautenbach, and
  Elham Roshanbin.
\newblock Bounds on the burning number.
\newblock {\em Discrete Applied Mathematics}, 235:16--22, 2018.

\bibitem[BEKM19]{bonato2019improved}
Anthony Bonato, Sean English, Bill Kay, and Daniel Moghbel.
\newblock Improved bounds for burning fence graphs.
\newblock {\em arXiv preprint arXiv:1911.01342}, 2019.

\bibitem[BFJ{\etalchar{+}}12]{bond201261}
Robert~M Bond, Christopher~J Fariss, Jason~J Jones, Adam~DI Kramer, Cameron
  Marlow, Jaime~E Settle, and James~H Fowler.
\newblock A 61-million-person experiment in social influence and political
  mobilization.
\newblock {\em Nature}, 489(7415):295, 2012.

\bibitem[BJR14]{bonato2014burning}
Anthony Bonato, Jeannette Janssen, and Elham Roshanbin.
\newblock Burning a graph as a model of social contagion.
\newblock In {\em International Workshop on Algorithms and Models for the
  Web-Graph}, pages 13--22. Springer, 2014.

\bibitem[BJR16]{bonato2016burn}
Anthony Bonato, Jeannette Janssen, and Elham Roshanbin.
\newblock How to burn a graph.
\newblock {\em Internet Mathematics}, 12(1-2):85--100, 2016.

\bibitem[BL19]{bonato2019bounds}
Anthony Bonato and Thomas Lidbetter.
\newblock Bounds on the burning numbers of spiders and path-forests.
\newblock {\em Theoretical Computer Science}, 794:12--19, 2019.

\bibitem[BSW20]{berry2020scanning}
Vincent Berry, Celine Scornavacca, and Mathias Weller.
\newblock Scanning phylogenetic networks is np-hard.
\newblock 2020.

\bibitem[CFK{\etalchar{+}}15]{cygan2015parameterized}
Marek Cygan, Fedor~V Fomin, {\L}ukasz Kowalik, Daniel Lokshtanov, D{\'a}niel
  Marx, Marcin Pilipczuk, Micha{\l} Pilipczuk, and Saket Saurabh.
\newblock {\em Parameterized algorithms}, volume~4.
\newblock Springer, 2015.

\bibitem[HWW08]{hulett2008multigraph}
Heather Hulett, Todd~G Will, and Gerhard~J Woeginger.
\newblock Multigraph realizations of degree sequences: Maximization is easy,
  minimization is hard.
\newblock {\em Operations Research Letters}, 36(5):594--596, 2008.

\bibitem[JJE{\etalchar{+}}18]{janssen2018exploring}
Remie Janssen, Mark Jones, P{\'e}ter~L Erd{\H{o}}s, Leo Van~Iersel, and Celine
  Scornavacca.
\newblock Exploring the tiers of rooted phylogenetic network space using tail
  moves.
\newblock {\em Bulletin of mathematical biology}, 80(8):2177--2208, 2018.

\bibitem[KGH14]{kramer2014experimental}
Adam~DI Kramer, Jamie~E Guillory, and Jeffrey~T Hancock.
\newblock Experimental evidence of massive-scale emotional contagion through
  social networks.
\newblock {\em Proceedings of the National Academy of Sciences},
  111(24):8788--8790, 2014.

\bibitem[KMZ19]{kamali2019burning}
Shahin Kamali, Avery Miller, and Kenny Zhang.
\newblock Burning two worlds: algorithms for burning dense and tree-like
  graphs.
\newblock {\em arXiv preprint arXiv:1909.00530}, 2019.

\bibitem[KR19]{kare2019parameterized}
Anjeneya~Swami Kare and I~Vinod Reddy.
\newblock Parameterized algorithms for graph burning problem.
\newblock In {\em International Workshop on Combinatorial Algorithms}, pages
  304--314. Springer, 2019.

\bibitem[LL16]{land2016upper}
Max~R Land and Linyuan Lu.
\newblock An upper bound on the burning number of graphs.
\newblock In {\em International Workshop on Algorithms and Models for the
  Web-Graph}, pages 1--8. Springer, 2016.

\end{thebibliography}

\end{document}